\documentclass[a4paper,reqno,12pt,oneside]{amsart}
\usepackage{amsmath,amsrefs,geometry}
\usepackage{color}

\numberwithin{equation}{section}

\DeclareMathOperator{\Ric}{Ric}
\DeclareMathOperator{\Rm}{Rm}

\DeclareMathOperator{\Weyl}{Weyl}
\DeclareMathOperator{\E}{E}

\DeclareMathOperator{\divergence}{div}
\DeclareMathOperator{\Vol}{Vol}
\DeclareMathOperator{\Eucl}{Eucl}
\newcommand{\<}{\left\langle}
\renewcommand{\>}{\right\rangle}
\renewcommand{\(}{\left(}
\renewcommand{\)}{\right)}
\renewcommand{\[}{\left[}
\renewcommand{\]}{\right]}
\newcommand{\eps}{\varepsilon}

\newtheorem{theorem}{Theorem}[section]
\newtheorem{proposition}[theorem]{Proposition}
\newtheorem{lemma}[theorem]{Lemma}

\begin{document}

\thanks{{\it Date:} October 1st, 2012}

\title[Blow-up solutions for linear perturbations of the Yamabe equation]{Blow-up solutions for linear perturbations of\\the Yamabe equation}

\author{Pierpaolo Esposito}
\address{Pierpaolo Esposito, Dipartimento di Matematica, Universit\`a degli Studi ``Roma Tre'',\,Largo S. Leonardo Murialdo 1, 00146 Roma, Italy}
\email{esposito@mat.uniroma3.it}

\author{Angela Pistoia}
\address{Angela Pistoia, Dipartimento SBAI, Universit\`a di Roma ``La Sapienza'', via Antonio Scarpa 16, 00161 Roma, Italy}
\email{pistoia@dmmm.uniroma1.it}

\author{J\'er\^ome V\'etois}
\address{J\'er\^ome V\'etois, Universit\'e de Nice - Sophia Antipolis, Laboratoire J.-A. Dieudonn\'e, UMR CNRS-UNS 7351, Parc Valrose, 06108 Nice Cedex 2, France}
\email{vetois@unice.fr}

\begin{abstract}
For a smooth, compact Riemannian manifold $(M,g)$ of dimension $N \ge3$, we are interested in the critical equation
$$\Delta_g u+\({N-2\over 4(N-1)}S_g+\eps   h\)u=u^{\frac{N+2}{N-2}}\quad\text{in}\ M\,,\quad u>0\quad\text{in}\ M\,,$$
where $\Delta_g$ is the Laplace--Beltrami operator, $S_g$ is the Scalar curvature of $(M,g)$, $h\in C^{0,\alpha}\(M\)$, and $\eps  $ is a small parameter.
\end{abstract}

\maketitle

\section{Introduction}

Letting $\(M,g\)$ be a smooth, compact Riemannian $N$--manifold, $N\geq 3$, we consider the solutions $u\in C^{2,\alpha}$ of the problem
\begin{equation}\label{Eq01}
\Delta_g u+\kappa u=c u^p,\quad u>0\quad\text{in }M\,,
\end{equation}
where $\Delta_g:=-\divergence_g\nabla$ is the Laplace--Beltrami operator, $\kappa\in C^{0,\alpha}\(M\)$, $\alpha\in\(0,1\)$, $c \in \mathbb{R}$, and $p>1$. 

\medskip 
When $\kappa=\alpha_N S_g$  and $p=2^*-1$, where $\alpha_N:=\frac{N-2}{4(N-1)}$, $S_g$ is the Scalar curvature of $(M,g)$ and $2^*:=\frac{2N}{N-2}$ is the critical Sobolev exponent, equation \eqref{Eq01} reads as 
\begin{equation}\label{Yam}
\Delta_g u+\frac{N-2}{4(N-1)} S_g u=c u^{\frac{N+2}{N-2}},\quad u>0\quad\text{in }M\,,
\end{equation}
and is referred to in the literature as the Yamabe problem. The constant $c$ can be restricted to the values $-1/1$ or $0$ depending on whether the {\it Yamabe invariant} of $(M,g)$, namely 
$$\mu_g(M)=\inf_{\widetilde g\in\[g\]}\(\Vol_{\widetilde g}\(M\)^{\frac{2-N}{N}}\int_M S_{\widetilde g}dv_{\widetilde g}\)$$
has negative/positive sign or vanishes, respectively, where $[g]=\{\phi g:\, \phi \in C^\infty(M),\, \phi>0 \}$ is the conformal class of $g$ and $\Vol_{\widetilde g}\(M\)$ is the volume of the manifold $\(M,\widetilde g\)$. If $u$ is a solution of \eqref{Yam}, then the metric $\widetilde g=u^{4/\(N-2\)} g$ has constant Scalar curvature and belongs to $[g]$. 

\medskip 
The Yamabe problem, raised by H. Yamabe~\cite{Yam} in '60, was firstly solved by Trudinger~\cite{Tru} when $\mu_g(M)\leq 0$. In this case, the solution is unique (up to a normalization when $\mu_g(M)=0$). In general, a solution of \eqref{Yam} can be found by a direct constrained minimization method. As shown by Aubin~\cite{Aub2}, the inequality 
\begin{equation} \label{keyineq}
\mu_g(M)<\mu_{g_0}(\mathbb{S}^N),
\end{equation} 
where $(\mathbb{S}^N, g_0)$ is the round sphere, is the key ingredient to show compactness of minimizing sequences, a non-trivial fact in view of the non-compactness of the Sobolev embedding $H_1^2(M) \hookrightarrow L^{2^*}(M)$. $(\mathbb{S}^N,g_0)$ has already constant Scalar curvature. For manifolds $(M,g)$ which are not conformally equivalent to $(\mathbb{S}^N,g_0)$ ($(M,g)\not=(\mathbb{S}^N,g_0)$ for short) with $\mu_g(M)>0$, the Yamabe equation \eqref{Yam} has been solved via \eqref{keyineq} by:  
\begin{itemize}
\item Aubin~\cite{Aub2} in the non-locally conformally flat case with $N\geq 6$, by exploiting the non-vanishing of the Weyl curvature tensor $\Weyl_g$ of $(M,g)$ in the construction of local test functions;  
\item Schoen~\cite{Sch1} when either $N=3,4,5$ or $(M,g) \not= (\mathbb{S}^N,g_0)$ is locally conformally flat, by exploiting the Positive Mass Theorem by Schoen--Yau~\cites{SchYau1,SchYau2} in the construction of global test functions
\end{itemize}
(see also Lee--Parker~\cite{LeePar} for a unified approach).

\medskip
From now on, we restrict our attention to the case where $\(M,g\)$ has {\it positive Yamabe invariant} $\mu_g(M)>0$. When $(M,g) \not= (\mathbb{S}^N,g_0)$, Schoen~\cite{Sch2} addressed the question of the compactness of Yamabe metrics, and he proved the compactness to be true in the locally conformally flat case~\cite{Sch2}. Recently, compactness of Yamabe metrics has been proved to be true for a general manifold $(M,g)\not=(\mathbb{S}^N,g_0)$ of dimension $N\le24$ by Khuri--Marques--Schoen~\cite{KhuMarSch}. Unexpectedly, compactness of Yamabe metrics has revealed to be false in general in dimensions $N\ge25$ by Brendle~\cites{Bre} and Brendle--Marques~\cite{BreMar}. Previous contributions where the compactness of Yamabe metrics is proved in lower dimensions are by Li--Zhu~\cite{LiZhu} ($N=3$), Druet~\cite{Dru2} ($N\leq 5$), Marques~\cite{Mar} ($N\leq 7$), and Li--Zhang ~\cites{LiZha1,LiZha2,LiZha3} ($N \leq 11$). In all these results, it is shown that sequences of solutions $\(u_k\)_{k\in\mathbb{N}}$ of (\ref{Eq01}) with $\kappa\equiv\alpha_N S_g$, $c=1$, and exponents $\(p_k\)_{k\in\mathbb{N}}$ in $[1+\varepsilon_0,2^*-1]$, $\varepsilon_0>0$ fixed, are pre-compact in $C^{2,\alpha}(M)$, $\alpha\in\(0,1\)$.

\medskip 
When $\kappa\not\equiv \alpha_N S_g$, the situation is different. When $\kappa<\alpha_N S_g$, Druet~\cites{Dru1,Dru2} (see also Druet--Hebey~\cite{DruHeb3} and Druet--Hebey--V\'etois~\cite{DruHebVet}) proved that compactness does hold for equation \eqref{Eq01} with $c=1$ and exponents $p$ in the range $\[1+\varepsilon_0,2^*-1\]$, for all dimensions $N\ge3$ (in case $N=3$, it is possible to write a more refined condition on the mass, see Li--Zhu~\cite{LiZhu}). As shown in Micheletti--Pistoia--V\'etois~\cite{MicPisVet} and Pistoia--V\'etois~\cite{PisVet}, in dimensions $N\ge4$, such a compactness result does not hold when $\kappa\(\xi_0\)>\alpha_n S_g\(\xi_0\)$ at some point $\xi_0\in M$ with a nondegeneracy assumption at $\xi_0$, and, see~\cite{MicPisVet}, compactness does not hold either in the supercritical range $p>2^*-1$ when $\kappa\(\xi_0\)<\alpha_N S_g\(\xi_0\)$ at some point $\xi_0\in M$. We also refer to Robert--V\'etois~\cite{RobVet}*{Theorem~2.3} where a special non-compactness result is obtained in dimension $N=6$ for potentials $\kappa>\alpha_N S_g$ (see also Druet~\cite{Dru1} and Druet--Hebey~\cites{DruHeb1,DruHeb2} in case of $\(M,g\)=(\mathbb{S}^N,g_0)$ with $N=6$). In the locally conformally flat case with $N\ge4$, Hebey--Vaugon~\cite{HebVau} proved that there always exists $\widetilde{g}\in\[g\]$ such that the equation $\Delta_{\widetilde{g}}u+\alpha_N\max_M(S_{\widetilde{g}})u=u^{2^*-1}$ in $M$ is not compact. In case $\(M,g\)=(\mathbb{S}^N,g_0)$ with $N\ge5$ and when $(\kappa-\alpha_N S_g)$ is a positive constant, Chen--Wei--Yan~\cite{ChenWeiYan} proved that equation \eqref{Eq01} with $c=1$ and $p=2^*-1$ is not compact (see also the constructions by Hebey--Wei~\cite{HebWei} in case $N=3$).

\medskip 
When the potential $\kappa$ varies, for manifolds $(M,g) \not= (\mathbb{S}^N,g_0)$ with $\mu_g(M)>0$, Druet~\cite{Dru2} (see also Druet--Hebey~\cite{DruHeb4}) proved that sequences of solutions $\(u_k\)_{k\in\mathbb{N}}$ of (\ref{Eq01}) with $c=1$, exponents $\(p_k\)_{k\in\mathbb{N}}$ in $[1+\varepsilon_0,2^*-1]$, and potentials $\(\kappa_k\)_{k\in\mathbb{N}}$, are pre-compact in $C^{2,\alpha}(M)$, $\alpha\in\(0,1\)$, when $n=3,4,5$ provided that $\kappa_k \leq \alpha_n S_g$. The same result is strongly expected to be true in the locally conformally flat case and generally for $N\leq 24$.

\medskip 
The aim of the paper is to investigate the effect of positive perturbations of the geometric potential by exhibiting the failure of compactness properties for the equation
\begin{equation}\label{Eq1}
\Delta_gu+(\alpha_N S_g+\varepsilon h) u=u^{2^*-1},\quad u>0\quad\text{in }M\,,
\end{equation}
where $h\in C^{0,\alpha}\(M\)$, $\alpha\in\(0,1\)$, with $\max_M h>0$ and $\varepsilon>0$ is a small parameter. 

\medskip 
A family $\(u_\varepsilon\)_\varepsilon$ of solutions to equation \eqref{Eq1} is said to {\it blow up} at some point $\xi_0\in M$ if there holds $\sup_U u_\varepsilon \to+\infty$ as $\varepsilon\to0$, for all neighborhoods $U$ of $\xi_0$ in $M$. 
Letting 
$$E(\xi):=\frac{h(\xi)}{\left|\Weyl_g(\xi)\right|_g}\,,$$
our main result is:

\begin{theorem}\label{ThTh}
Let $(M,g) \not=(\mathbb{S}^N,g_0)$ be a smooth, compact, non-locally conformally flat Riemannian manifold with $N\ge6$ and $\mu_g(M)>0$. Let $h\in C^{0,\alpha}\(M\)$, $\alpha\in\(0,1\)$, so that $\max_M h>0$ and $\inf\{|\Weyl_g(x)|_g\,:\,h(x)>0\}>0$. Then for $\varepsilon>0$ small, equation \eqref{Eq1} has a solution $u_\varepsilon$ such that the family $\(u_\varepsilon\)_\varepsilon$ blows up, up to a sub-sequence, as $\varepsilon \to 0$ at some point $\xi_0$ so that $E(\xi_0)=\max_M E$ . 
\end{theorem}

Introducing the ``reduced energy" $\widetilde E:(0,\infty)\times M \to \mathbb{R}$ defined as 
$$\widetilde E(d,\xi)=c_2 d^2 h(\xi)-c_3 d^4 \left|\Weyl_g(\xi)\right|_g^2$$
with $c_2,c_3>0$, Theorem~\ref{ThTh} is an easy consequence of the following more general result:

\begin{theorem}\label{Th}
Let $(M,g) \not=(\mathbb{S}^N,g_0)$ be a smooth, compact, non-locally conformally flat Riemannian manifold with $N \ge6$ and $\mu_g(M)>0$, and $h\in C^{0,\alpha}\(M\)$, $\alpha\in\(0,1\)$. Assume that there exists a $C^0$--{\it stable critical set} $\mathcal{D} \subset (0,\infty) \times M$ of $\widetilde E$. Then for $\varepsilon>0$ small, equation \eqref{Eq1} has a solution $u_\varepsilon$ such that the family $\(u_\varepsilon\)_\varepsilon$ blows up, up to a sub-sequence, at some $\xi_0 \in \pi(\mathcal{D})$, where $\pi: (0,\infty)\times M \to M$ is the projection operator onto the second component. 
\end{theorem}

According to Li~\cite{Li0}, we say that a compact set $\mathcal{D} \subset (0,\infty)\times M$ of critical points of $\widetilde E$ is a $C^0$--{\it stable critical set} of $\widetilde E$ if for any compact neighborhood $U$ of $\mathcal{D}$ in $(0,\infty) \times M$, there exists $\delta>0$ such that, if $\mathcal{J}\in C^1\(U\)$ and $\|\mathcal{J}-\widetilde E\|_{C^0(U)}\leq \delta$, then $\mathcal{J}$ has at least one critical point in $U$. 

\medskip 
Given $\xi \in M$ so that $h(\xi)>0$, define $d(\xi)$ as
$$d(\xi)=\left(\frac{c_2 h(\xi)}{2 c_3 \left|\Weyl_g\(\xi\)\right|_g^2}\right)^{1/2}$$
with the convention that $d(\xi)=+\infty$ if $\Weyl_g(\xi)=0$. Given $\xi \in M$ with $h(\xi)>0$, the function $\widetilde E$ is increasing for $d \in (0,d(\xi))$ and, if $d(\xi)<+\infty$, achieves its global maximum in $d$ at $d(\xi)$. Since 
$$\widetilde E(d(\xi),\xi)=\frac{c_2^2 h^2(\xi)}{4 c_3 \left|\Weyl_g\(\xi\)\right|_g^2}=\frac{c_2^2}{4 c_3} E(\xi)^2,$$
in order to derive Theorem~\ref{ThTh}, the set $\mathcal{D}$ in Theorem~\ref{Th} is constructed as 
$$\mathcal{D}=\{(d(\xi),\xi): \, \xi\in M\hbox{ s.t. } E(\xi)=\max_M E\},$$ 
which is clearly a $C^0$--{\it stable critical set} of $\widetilde E$. Since $d(\xi)$ is a maximum point of $\widetilde E$ in $d$, neither minimum points of $E$, nor saddle points of $E$ can provide any $C^0$--{\it stable critical set} of $\widetilde E$. 

\medskip 
Let us finally compare problem \eqref{Eq1} with its Euclidean counter-part on a smooth bounded domain $\Omega \subset \mathbb{R}^N$, $N\ge4$, with homogeneous Dirichlet boundary condition:
\begin{equation}\label{bn}
\Delta_{\Eucl} u+\lambda u=u^{2^*-1}  \ \hbox{in}\ \Omega,\quad u>0 \ \hbox{in}\ \Omega,\quad u=\ \hbox{on}\ \partial\Omega.
\end{equation}
For $\lambda \geq 0$, a direct minimization method (for the corresponding Rayleigh quotient) never gives rise to any solution of \eqref{bn}, and no solutions exist at all if $\Omega$ is star-shaped as shown by Poho{\v{z}}aev~\cite{poho}. Moreover, following the arguments developed by Ben Ayed--El Mehdi--Grossi--Rey~\cite{begr}, problem \eqref{bn} has never any solution with a single blow-up point as $\lambda \to 0^+$. The effect of the geometry, which is crucial to provide a solution for the Yamabe problem (corresponding to $\lambda=0$ in (\ref{bn})) by minimization, is also relevant to producing solutions of \eqref{Eq1} (corresponding to $\lambda \to 0^+$ in (\ref{bn})) with a single blow-up point as stated in Theorem~\ref{ThTh}.\\
When $\lambda<0$, solutions of \eqref{bn} can be found by direct minimization as shown by Brezis--Nirenberg~\cite{bn}, and exhibit a single blow-up point as $\lambda \to 0^-$ as shown by Han~\cite{h}, in contrast with the compactness property proved by Druet~\cites{Dru1,Dru2}. Solutions of \eqref{bn} with a single blow-up point, see Rey~\cites{Rey,rey2}, and with multiple blow-up points, see Bahri--Li--Rey~\cites{blr} and Musso--Pistoia~\cite{mupi4}, as $\lambda \to 0^-$ have been constructed in a very general way. 

\medskip 
We attack the existence issue of blowing-up solutions by a perturbative method, referred to in the literature as the non-linear Lyapunov--Schmidt reduction. Such a method is well known and the main point is to produce a suitable ansatz for the solutions. In the non-locally conformally flat case with $N\geq 6$ the basic ansatz is like in Aubin~\cite{Aub2}, but, see Section~\ref{Sec2}, needs to be slightly corrected via linearization so to account for the local geometry. A similar idea has been used for the prescribed $Q-$curvature problem by Pistoia--Vaira~\cite{piva}, the fourth-order analogue of the Yamabe problem. An alternative and more geometrical approach can be devised based on the conformal covariance of $\Delta_g+\alpha_N S_g$. The main point is to allow the metric $g$ to vary in the conformal class so to gain flatness at each point $\xi \in M$, and this approach allows us, see Esposito--Pistoia--V\'etois~\cite{EspPisVet}, to cover in an unified way also the remaining cases $N=4,5$ or $(M,g)$ locally conformally flat with $N\geq 6$ (the case $N=3$ is always excluded by the compactness result of Li--Zhu~\cite{LiZhu}). The aim of this paper is at the same time to advertise the general result contained in~\cite{EspPisVet},  and to provide a simpler and more intuitive proof in a special case. Thanks to the solvability theory of the linearized operator, we are led to study critical points of a finite-dimensional functional $\mathcal{J}_\varepsilon$, and a key step is to obtain in Section~\ref{Sec3} an asymptotic expansion of $\mathcal{J}_\varepsilon$ by identifying the ``reduced energy" $\widetilde E$ as the main order term. In Section~\ref{Sec4}, we describe the main steps of the non-linear Lyapunov-Schmidt reduction, and we deduce our general result Theorem~\ref{Th}.

\section{The correcting term towards an improved ansatz}\label{Sec2}

Letting 
\begin{equation}\label{defU}
U (r)= \({ \sqrt{N(N-2)} \over 1+r^2}\)^{N-2\over2},
\end{equation}
we aim to solve
\begin{equation}\label{equationV}
\Delta V+pU^{p-1} V=\frac{1}{3} \sum_{i,j=1}^N R_{ij}(\xi) \frac{y^i y^j}{|y|} \partial_r U+\alpha_N S_g(\xi) U\,,
\end{equation}
where $p=\frac{N+2}{N-2}$ and $R_{ij}$ are the components of the Ricci tensor $\Ric_g$ of $(M,g)$ in geodesic coordinates. Here, $\Delta=\displaystyle \sum_{i=1}^N \frac{\partial^2}{\partial y_i^2}$ is the Euclidean laplacian with the standard sign convention, and $U(|y|)$ is the unique positive radial solution of $-\Delta U=U^p$ with $U(0)=\max\limits_{\mathbb{R}^N} U=[N(N-2)]^{\frac{N-2}{4}}$.

\medskip 
Since $S_g(\xi)=\displaystyle \sum_{i=1}^N R_{ii}(\xi)$,  a straightforward computation shows that 
\begin{equation} \label{defV}
V(y)=[N(N-2)]^{\frac{N-2}{4}} \( \frac{|y|^2+3}{12(1+|y|^2)^{\frac{N}{2}}} \sum_{i,j=1}^N R_{ij}(\xi)  
 y^i y^j-\frac{S_g(\xi)}{24(N-1) }   \frac{|y|^4+3} {(1+|y|^2)^{\frac{N}{2}} }\)
\end{equation}
is a solution of \eqref{equationV} as we were searching for. 

\medskip  
Let $0<r_0<i_g(M)$, where $i_g(M)$ is the injectivity radius of $(M,g)$. Take $\chi$ a smooth cutoff function  such that $0\le\chi\le1$ in $\mathbb{R}$, $\chi\equiv1$ in $ \[-r_0/2,r_0/2 \]$, and $\chi\equiv0$ out of $ \[-r_0,r_0 \]$. For any point $\xi$ in $M$ and for any positive real number $\mu$, we define the functions $\mathcal{U}_{\mu,\xi}$ and  $\mathcal{V}_{\mu,\xi}$ on $M$ by
$$\mathcal{U}_{\mu,\xi} (z )=\chi \(d_g (z,\xi ) \)U_\mu\( d_g (z,\xi )\)\,,\:\:
\mathcal{V}_{\mu,\xi} (z )=\chi \(d_g (z,\xi ) \)V_\mu\( \exp_\xi^{-1}\(z\)\),$$
where $d_g$ is the geodesic distance in $(M,g)$ and $\exp_\xi^{-1}$ is the geodesic coordinate system. Here, $U_\mu$ and $V_\mu$ are defined as
$$U_{\mu }(x)=\mu^{-{N-2\over2}}U\({x \over\mu}\) \,,\:\:\: V_{\mu }(x)=\mu^{-{N-2\over2}}V\({x \over\mu}\),$$
obtained by scaling $U$ and $V$ in (\ref{defU}) and (\ref{defV}), respectively. Since $\mu_g(M)>0$ implies the coercivity of the conformal laplacian $\Delta_g+\alpha_N S_g$, let $i^*:L^{\frac{2N}{N+2}}\(M\)\rightarrow H^1_g(M)$ be the bounded operator defined as follows: the function $u=i^*(w)$ is the unique solution in $H^1_g\(M\)$ of the equation $\Delta_g u+ \alpha_N S_g u=w$ in $M$. Problem \eqref{Eq1} re-writes as
\begin{equation}\label{Eq1b}
u=i^*\[u_+^p -\eps   h u\]\,,\end{equation}
and we look for solutions of \eqref{Eq1b} in the form
\begin{equation}\label{Eq17}
u_\eps  (z)=\mathcal{W}_{\mu,\xi} (z ) +\phi_{\eps   }(z),\quad  \mathcal{W}_{\mu,\xi} =\mathcal{U}_{\mu,\xi}  +\mu^2\mathcal{V}_{\mu,\xi}\,,
\end{equation}
where $\xi\in M$, $\mu>0$ is small and $\phi_{\eps}$ is a small remainder term.
 
\medskip
First of all, we introduce the error term
\begin{equation}\label{rmx}
\mathcal{R}_{\mu,\xi}= \mathcal{W}_{\mu,\xi}  -i^*\[\(\mathcal{W}_{\mu,\xi}\)_+^p- \eps   h \mathcal{W}_{\mu,\xi} \]\,.
\end{equation}
We want to point out that the choice of the ansatz in \eqref{Eq17} with the extra term $\mathcal{V}_{\mu,\xi}$ is motivated by the need that the error term has to be small enough. Indeed, the error term is estimated as follows.
  
\begin{lemma}\label{Lem2}
Let $N\ge 6.$  There exists a positive constant $C_0>0$ such that for any $\mu$ small and $\xi$ in $M$ there holds
\begin{equation}\label{Lem2Eq1}
\left\|\mathcal{R}_{\mu,\xi}\right\|\le  C_0 \left\{\begin{array}{ll} \mu^{\frac{N-2}{2}}+\eps \mu^2 |\ln \mu|^{\frac{2}{3}} &\hbox{if } N=6\\
\mu^{\frac{N-2}{2}}+\eps \mu^2&\hbox{if } N=7\\
\mu^3 |\ln \mu|^{\frac{5}{8}}+\eps \mu^2&\hbox{if }N=8\\
\mu^3+\eps \mu^2&\hbox{if }N=9\\
\mu^{2\frac{N+2}{N-2}}+\eps \mu^2&\hbox{if }N\geq 10. \end{array} \right.
\end{equation}
\end{lemma}

\begin{proof}
It is enough to estimate the $L^{\frac{2N}{N+2}}$--norm of
$$\Delta_g \mathcal{W}_{\mu,\xi}+(\alpha_N S_g+\eps h)\mathcal{W}_{\mu,\xi}  - \(\mathcal{W}_{\mu,\xi}\)_+^p.$$
Since $\mathcal{U}_{\mu,\xi} \circ \exp_\xi$ is radially symmetric in $B_0(r_0)$, we have that
$$\Delta_g \mathcal{U}_{\mu,\xi} \(\exp_\xi x\)=-\Delta \( \mathcal{U}_{\mu,\xi} \circ \exp_\xi\) (x) -\frac{1}{2}\partial_r (\ln |g|) \partial_r \( \mathcal{U}_{\mu,\xi} \circ \exp_\xi\) (x),$$
where $|g|:=\hbox{det }g$. In geodesic coordinates, we have the Taylor expansion
\begin{equation} \label{Taylorg}
|g|=1-\frac{1}{3} \sum_{i,j=1}^N R_{ij}(\xi) x^i x^j+\operatorname{O}(|x|^3)
\end{equation}
(see for example Lee--Parker~\cite{LeePar}), yielding to
\begin{align}
&\Delta_g \mathcal{U}_{\mu,\xi} \(\exp_\xi x\)=-\chi(|x|) \Delta U_\mu (x) +\frac{\chi(|x|)}{3} \sum_{i,j=1}^N \frac{R_{ij}(\xi) x^i x^j}{|x|} \partial_r U_\mu(x) \nonumber\\
&\qquad\qquad\qquad\qquad+\operatorname{O}\(\mu^{\frac{N-2}{2}}+|x|^2 |\nabla U_\mu|\) \nonumber \\
&\qquad\qquad\quad= \mathcal{U}_{\mu,\xi}^p \(\exp_\xi x\) +\frac{\chi(|x|)}{3} \sum_{i,j=1}^N \frac{R_{ij}(\xi) x^i x^j}{|x|} \partial_r U_\mu(x)+\operatorname{O}\(\mu^{\frac{N-2}{2}}+|x|^2 |\nabla U_\mu|\) \label{formula1}
\end{align}
in view of $-\Delta U_\mu=U_\mu^p$. Similarly, we have that
$$\Delta_g \mathcal{V}_{\mu,\xi} \(\exp_\xi x\)=-\chi(|x|) \Delta V_\mu (x) +\operatorname{O}\(\mu^{\frac{N-6}{2}}+|x| |\nabla V_\mu|\).
$$
Since by \eqref{equationV} we have that
\begin{equation}\label{equationVbis}
\Delta (\mu^2 V_\mu)+pU_\mu^{p-1} (\mu^2 V_\mu)=\frac{1}{3} \sum_{i,j=1}^N R_{ij}(\xi) \frac{x^i x^j}{|x|} \partial_r U_\mu+\alpha_N S_g(\xi) U_\mu,
\end{equation}
by \eqref{formula1}--\eqref{equationVbis} we get that
\begin{eqnarray}
\|\Delta_g \mathcal{W}_{\mu,\xi}+\alpha_N S_g \mathcal{W}_{\mu,\xi}- \mathcal{U}_{\mu,\xi}^p-p\mu^2 \mathcal{U}_{\mu,\xi}^{p-1} \mathcal{V}_{\mu,\xi} \|_{L^{\frac{2N}{N+2}}(M)}
=\left\{\begin{array}{ll} \operatorname{O}\( \mu^{\frac{N-2}{2}} \) &\hbox{if } N=6,7\\
\operatorname{O}\(\mu^3 |\ln \mu|^{\frac{5}{8}}\) &\hbox{if }N=8\\
\operatorname{O}\( \mu^3 \) &\hbox{if }N\geq 9. \end{array} \right.
\label{formula3}
\end{eqnarray}
Since
$$\|h\mathcal{W}_{\mu,\xi}\|_{L^{\frac{2N}{N+2}}(M)}=\left\{\begin{array}{ll} 
\operatorname{O}\(\mu^2 |\ln \mu|^{\frac{2}{3}}\) &\hbox{if }N=6\\
\operatorname{O}\( \mu^2 \) &\hbox{if }N\geq 7 \end{array} \right.$$
and
$$\|\(\mathcal{W}_{\mu,\xi}\)_+^p - \mathcal{U}_{\mu,\xi}^p-p \ \mathcal{U}_{\mu,\xi}^{p-1} \(\mu^2 \mathcal{V}_{\mu,\xi}\) \|_{L^{\frac{2N}{N+2}}(M)}=\left\{\begin{array}{ll} 
\operatorname{O}\(\mu^4 |\ln \mu|^{\frac{2}{3}}\) &\hbox{if }N=6\\
\operatorname{O}\( \mu^{2\frac{N+2}{N-2}} \) &\hbox{if }N\geq 7 \end{array} \right.
$$
in view of $|(a+b)_+^p-a^p-p a^{p-1}b|=\operatorname{O}(|b|^p)$ for all $a>0$ and $b\in \mathbb{R}$,
by \eqref{formula3} we deduce the validity of \eqref{Lem2Eq1}.
\end{proof}

\section{The reduced energy}\label{Sec3}

Introduce the Euler-Lagrange functional $J_\eps  :{\rm H}^1_g(M)\to \mathbb{R}$ corresponding to equation \eqref{Eq1}:
$$J_\eps  (u):={1\over2}\int_M|\nabla u|_g^2 dv_g+{1\over2}\int_M\(\alpha_NS_g+\eps   h\) u ^2 dv_g-{1\over p+1} \int_M  u_+ ^{p+1} dv_g\,.$$
The aim is to find an asymptotic expansion of $J_\varepsilon\(\mathcal{W}_{\mu,\xi}\)$. We have that:

\begin{proposition}\label{Pr2} 
The following expansions do hold as $\epsilon,\,\mu  \to 0$:
\begin{eqnarray}\label{reduceden0}
J_\varepsilon\(\mathcal{W}_{\mu,\xi}\)
=\frac{K_6^{-6}}{6} +\frac{4}{5} \omega_5
 \left|\Weyl_g(\xi)\right|^2_g  \mu^4 \ln \mu+\frac{5}{24}K_6^{-6} h(\xi) \eps  \mu^2+\operatorname{o}\(\mu^4\ln \mu +\eps \mu^2 \) 
\end{eqnarray}   
when $N=6$, and
\begin{multline}\label{reduceden}
J_\varepsilon\(\mathcal{W}_{\mu,\xi}\)
=\frac{K_N^{-N}}{N} -
\frac{K_N^{-N}}{24N(N-4)(N-6)} \left|\Weyl_g(\xi)\right|^2_g  \mu^4+\frac{2(N-1)K_N^{-N}h(\xi)}{N(N-2)(N-4)}\eps  \mu^2\\
+\operatorname{o}\(\mu^4+\eps \mu^2 \) 
\end{multline}   
when $N\geq 7$, uniformly with respect to $\xi \in M$, where $K_N$ is the best constant for the embedding of $D^{1,2}\(\mathbb{R}^N\)$ into $L^{2^*}\(\mathbb{R}^N\)$.
\end{proposition}

\begin{proof}
First, we have that
\begin{eqnarray}\label{energy1.2}
&&J_\eps  \(\mathcal{U}_{\mu,\xi}+\mu^2\mathcal{V}_{\mu,\xi}\)-J_\eps  \(\mathcal{U}_{\mu,\xi} \)=
\mu^2\int_M  \[\<\nabla  \mathcal{U}_{\mu,\xi},\nabla \mathcal{V}_{\mu,\xi}\>_g+\(\alpha_NS_g+\eps   h\) \mathcal{U}_{\mu,\xi} \mathcal{V}_{\mu,\xi} - \mathcal{U}_{\mu,\xi}   ^{p } \mathcal{V}_{\mu,\xi}\] \nonumber\\ 
&&\quad  +{1\over2}\mu^4 \int_M \[\left|\nabla  \mathcal{V}_{\mu,\xi} \right|_g^2   -p  \mathcal{U}_{\mu,\xi}^{p -1}
   \mathcal{V}_{\mu,\xi}^2\]  dv_g+ {1\over2}\mu^4 \int_M \(\alpha_N S_g+\eps   h\) \  \mathcal{V}_{\mu,\xi} ^2 dv_g\nonumber\\  
&&\quad-{1\over p+1} \int_M  \[\left(\mathcal{U}_{\mu,\xi}+\mu^2\mathcal{V}_{\mu,\xi} \)_+ ^{p+1}-  \mathcal{U}_{\mu,\xi}^{p+1}-(p+1)\mathcal{U}_{\mu,\xi}^{p }
  \mu^2\mathcal{V}_{\mu,\xi}-{1\over2}p(p+1)\mathcal{U}_{\mu,\xi}^{p -1}
  \mu^4\mathcal{V}_{\mu,\xi}^2\]
 dv_g\nonumber\\
 &&=
\mu^2 \int_M \[\Delta_g  \mathcal{U}_{\mu,\xi}    + \alpha_N S_g  \mathcal{U}_{\mu,\xi}   - \mathcal{U}_{\mu,\xi}   ^{p }\] \mathcal{V}_{\mu,\xi}  dv_g
+{1\over2}\mu^4\int_M \[\Delta_g  \mathcal{V}_{\mu,\xi}       - p\mathcal{U}_{\mu,\xi}   ^{p -1} \mathcal{V}_{\mu,\xi} \] 
\mathcal{V}_{\mu,\xi}  dv_g \nonumber \\ 
&&\quad+\left\{
\begin{array}{ll}
\operatorname{o}\(\mu^4 \ln \mu\) &\hbox{if } N=6\\
\operatorname{o}\(\mu^4\) &\hbox{if } N\geq 7 
\end{array}\right. 
\end{eqnarray}
as $\mu \to 0$, in view of
\begin{align*}
& \int_M  \left|\left(\mathcal{U}_{\mu,\xi}+\mu^2\mathcal{V}_{\mu,\xi} \)_+ ^{p+1}-  \mathcal{U}_{\mu,\xi}^{p+1}-(p+1)\mathcal{U}_{\mu,\xi}^{p }
  \mu^2\mathcal{V}_{\mu,\xi}-{1\over2}p(p+1)\mathcal{U}_{\mu,\xi}^{p -1}
  \mu^4\mathcal{V}_{\mu,\xi}^2\right|
 dv_g\nonumber\\
 &\qquad=  \operatorname{O}\( \mu^{\frac{4N}{N-2}} \int_M \left|\mathcal{V}_{\mu,\xi} \right| ^{\frac{2N}{N-2}}dv_g\)
 =\operatorname{o}\(\mu^4\)
 \end{align*}
and $\int_M  \mathcal{V}_{\mu,\xi} ^2 dv_g= \left\{
\begin{array}{ll}
\operatorname{O}\(1\) &\hbox{if } N=6\\
\operatorname{o}\(1\) &\hbox{if } N\geq 7 
\end{array}\right.$ as $\mu \to 0$. Now, observe that there holds 
\begin{eqnarray*}
&& \mu^2 \int_M \[\Delta_g  \mathcal{U}_{\mu,\xi}    + \alpha_N S_g  \mathcal{U}_{\mu,\xi}   - \mathcal{U}_{\mu,\xi}   ^{p }\] \mathcal{V}_{\mu,\xi}  dv_g
+\mu^4 \int_M \[\Delta_g  \mathcal{V}_{\mu,\xi}       - p\mathcal{U}_{\mu,\xi}   ^{p -1} \mathcal{V}_{\mu,\xi} \] 
\mathcal{V}_{\mu,\xi}  dv_g \\
&&\qquad=\mu^2 \int_M \[\Delta_g  \mathcal{W}_{\mu,\xi}    + \alpha_N S_g  \mathcal{W}_{\mu,\xi}   -\mathcal{U}_{\mu,\xi}^p-p\mu^2 \mathcal{U}_{\mu,\xi}^{p-1} \mathcal{V}_{\mu,\xi}\] \mathcal{V}_{\mu,\xi}  dv_g +\operatorname{O}\( \mu^4 \int_M  \mathcal{V}_{\mu,\xi} ^2dv_g\) \nonumber\\
&&\qquad=\left\{
\begin{array}{ll}
\operatorname{o}\(\mu^4 \ln \mu\) &\hbox{if } N=6\\
\operatorname{o}\(\mu^4\) &\hbox{if } N\geq 7 
\end{array}\right. \nonumber
\end{eqnarray*}
as $\mu \to 0$, in view of \eqref{formula3}. By (\ref{Taylorg}) and 
\begin{align*}
\Delta_g  \mathcal{V}_{\mu,\xi}(\exp_\xi x)&=-\Delta\(\mathcal{V}_{\mu,\xi}\circ \exp_\xi \)(x)\\
&\quad+\operatorname{O}\(|x|\left|\nabla\(\mathcal{V}_{\mu,\xi}\circ\exp_\xi\)\(x\)\right|+|x|^2\left|\nabla^2\(\mathcal{V}_{\mu,\xi}\circ\exp_\xi\)\(x\)\right|\)
\end{align*}
we deduce that
\begin{equation}\label{added1}
\int_M \[\Delta_g  \mathcal{V}_{\mu,\xi}-p\mathcal{U}_{\mu,\xi}^{p -1}\mathcal{V}_{\mu,\xi} \] 
\mathcal{V}_{\mu,\xi}dv_g=
-\int_{B_0(\frac{r_0}{2 \mu})} \(\Delta V+p U^{p -1} V\) Vdy+\left\{
\begin{array}{ll}
\operatorname{O}\(1\) &\hbox{if } N=6\\
\operatorname{o}\(1\) &\hbox{if } N\geq 7 
\end{array}\right.
\end{equation}
as $\mu \to 0$. By (\ref{energy1.2}) and \eqref{added1}, we get that
\begin{equation}\label{tttt}
J_\eps\(\mathcal{U}_{\mu,\xi}+\mu^2\mathcal{V}_{\mu,\xi}\)=J_\eps  \(\mathcal{U}_{\mu,\xi} \)
+{1\over2}\mu^4 \int_{B_0(\frac{r_0}{2 \mu})} \(\Delta V+p U^{p -1} V\) Vdy +\left\{
\begin{array}{ll}
\operatorname{o}\(\mu^4 \ln \mu\) &\hbox{if } N=6\\
\operatorname{o}\(\mu^4\) &\hbox{if } N\geq 7 
\end{array}\right. 
\end{equation}
as $\mu \to 0$. By (\ref{equationV})--(\ref{defV}) and easy symmetry properties we deduce that
\begin{align}\label{additionalterm}
& \int_{B_0(\frac{r_0}{2 \mu})} \(\Delta V+p U^{p -1} V\) Vdy \nonumber \\
&\quad= -\frac{[N(N-2)]^{\frac{N-2}{2}} (N-2)}{36} \int_{B_0(\frac{r_0}{2 \mu})} \(\sum_{i,j=1}^N R_{ij}(\xi) y^i y^j\)^2   \frac{|y|^2+3}{(1+|y|^2)^N}dy \nonumber \\
&+\frac{[N(N-2)]^{\frac{N-2}{2}} \alpha_N}{72 N(N-1)} S_g^2(\xi) \int_{B_0(\frac{r_0}{2 \mu})}   \frac{(7N-10)|y|^6+3(7N-8)|y|^4+3(7N-10)|y|^2-9N} {(1+|y|^2)^N} dy \nonumber\\
&\quad= -\frac{[N(N-2)]^{\frac{N-2}{2}} (N-2)}{36 } \int_{B_0(\frac{r_0}{2 \mu})} \sum_{i,j,k,s=1}^N E_{ij}(\xi) E_{ks}(\xi) y^i y^j  y^k y^s    \frac{|y|^2+3}{(1+|y|^2)^N}dy \nonumber \\
&\qquad-\omega_{N-1} \frac{[N(N-2)]^{\frac{N-2}{2}} (N-2)}{576 N^2(N-1)^2} S_g^2(\xi) \left[(N-2)(N-4)I_N^{\frac{N+4}{2}}+3(N^2-8N+8) I_N^{\frac{N+2}{2}} \right. \nonumber \\
&\qquad\left.-3N(7N-10)I_N^{\frac{N}{2}}+9N^2 I_N^{\frac{N-2}{2}}\right]+\operatorname{o}(1) 
\end{align}
as $\mu \to 0$, where the $E_{ij}$'s are the components of the traceless part $\E_g=\Ric_g-\frac{S_g}{N}g$ of the Ricci curvature $\Ric_g$ of $(M,g)$ in geodesic coordinates and 
$$I^q_p=\left\{\begin{aligned} & \int_0^{+\infty}\frac{r^q}{\(1+r\)^p}dr &\hbox{if }p-q>1\\
&\int_0^{\frac{r_0^2}{4 \mu^2}}\frac{r^q}{\(1+r\)^p}dr &\hbox{if }p-q\leq 1.
\end{aligned} \right. $$
Since integration by parts yields to 
\begin{equation}\label{AubinEq11}
I^q_{p+1}=\frac{p-q-1}{p}I^q_p\quad\text{and}\quad I^{q+1}_{p+1}=\frac{q+1}{p-q-1}I^q_{p+1}
\end{equation}
as soon as $p-q>1$, we have that
\begin{equation}\label{llp}
I_N^{\frac{N}{2}}=\frac{N}{N-2}I_N^{\frac{N-2}{2}}=\frac{N-4}{N+2}I_N^{\frac{N+2}{2}}\quad\text{and}\quad I_N^{\frac{N+4}{2}}=\left\{\begin{array}{ll}-2 \ln \mu+\operatorname{O}(1) &\hbox{if }N=6\\
\frac{(N+2)(N+4)}{(N-4)(N-6)}I_N^{\frac{N}{2}}& \hbox{if }N\geq 7
\end{array}\right.
\end{equation}
as $\mu \to 0$, and it can be easily checked that
\begin{equation}\label{AubinEq16}
I_N^{\frac{N}{2}}=\frac{N\omega_N}{2^{N-1}\(N-2\)\omega_{N-1}}=\frac{2K_N^{-N}}{ [N(N-2)]^{\frac{N-2}{2}} (N-2)^2\omega_{N-1}}
\end{equation}
(see Aubin~\cite{Aub2}). Since for all $i \not=j$ there holds 
$$\int_{S^{N-1}} (y^i)^4 d v_{g_0}=3 \int_{S^{N-1}} (y^i)^2(y^j)^2 d v_{g_0}= \frac{3}{N(N+2)}\int_{S^{N-1}} |y|^4 d v_{g_0}\,,$$
by \eqref{additionalterm} and \eqref{llp}--\eqref{AubinEq16} we deduce that
\begin{eqnarray}\label{llp1}
\int_{B_0(\frac{r_0}{2 \mu})} \(\Delta V+p U^{p -1} V\) Vdy =\frac{8}{3 } \omega_5 |E_g(\xi)|_g^2  \ln \mu +
 \frac{16}{225} \omega_5  S_g^2(\xi)  \ln \mu +\operatorname{O}(1)
\end{eqnarray}
if $N=6$, and
\begin{multline} \label{llp2}
\int_{B_0(\frac{r_0}{2 \mu})} \(\Delta V+p U^{p -1} V\) Vdy = -\frac{2N-7}{9 N(N-2)(N-4)(N-6) } K_N^{-N} |E_g(\xi)|_g^2\\
+ \frac{(N-2)(N-7)}{36 N^2(N-1)(N-4)(N-6)} K_N^{-N}S_g^2(\xi) +\operatorname{o}(1) 
\end{multline}
if $N\geq 7$. Inserting \eqref{llp1}--\eqref{llp2} into \eqref{tttt}, by Lemma~\ref{energy2} below we deduce the validity of \eqref{reduceden0}--\eqref{reduceden}.
\end{proof}

We are left with proving the following:

\begin{lemma}\label{energy2}
The following expansions do hold as $\epsilon,\,\mu  \to 0$:
\begin{align*}
J_\eps  \(U_{\mu,\xi}\)&= \frac{K_6^{-6}}{6}+\[\frac{4}{5}
\left|\Weyl_g\(\xi\)\right|^2_g -
\frac{4}{3 } |E_g(\xi)|_g^2  - \frac{8}{225} S_g^2(\xi) \] \omega_5 \mu^4 \ln \mu+\frac{5}{24}K_6^{-6} h(\xi) \eps  \mu^2  \\ 
&\quad +\operatorname{o}\(\mu^4 \ln \mu +\eps \mu^2 \)
\end{align*}
when $N=6$, and
\begin{align*}
J_\eps  \(U_{\mu,\xi}\)&=\frac{K_N^{-N}}{N}+\[-\frac{K_N^{-N}}{24N(N-4)(N-6)} \left|\Weyl_g(\xi)\right|^2_g\right.\nonumber\\ 
&\quad\left.+\frac{(2N-7)K_N^{-N}}{18N(N-2)(N-4)(N-6)} \left|\E_g(\xi)\right|^2_g -\frac{(N-2)(N-7)K_N^{-N}}{72N^2(N-1)(N-4)(N-6)}
S_g(\xi) ^2 \] \mu^4\nonumber\\ 
&\quad+\frac{2(N-1)K_N^{-N}}{N(N-2)(N-4)}h(\xi)\eps  \mu^2+\operatorname{o}\(\mu^4+\eps \mu^2 \)
\end{align*}
when $N\geq 7$, uniformly with respect to $\xi \in M$.
\end{lemma}

\begin{proof}
There hold
\begin{align}
&\frac{1}{\omega_{N-1}r^{N-1}}\int_{\partial B_\xi\(r\)}hd\sigma_g=h\(\xi\)+\operatorname{O}\(r \),\label{AubinEq3}\\
&\frac{1}{\omega_{N-1}r^{N-1}}\int_{\partial B_\xi\(r\)}S_g d\sigma_g=S_g \(\xi\)-\frac{1}{2N}\varLambda_g\(\xi\)r^2+\operatorname{O}\(r^4\),\label{AubinEq4}\\
&\frac{1}{\omega_{N-1}r^{N-1}}\int_{\partial B_\xi\(r\)}d\sigma_g=1-\frac{1}{6N}S_g\(\xi\)r^2+A_g\(\xi\)r^4+\operatorname{O}\(r^5\),\label{AubinEq5}
\end{align}
as $r\to0$, uniformly with respect to $\xi$, where $d\sigma_g$ is the volume element of $\partial B_\xi\(r\)$, $\omega_{N-1}$ is the volume of the unit $\(N-1\)$--sphere, and where (see \eqref{AubinEq8}--\eqref{AubinEq9})
 \begin{equation}\label{AubinEq6}
\varLambda_g\(\xi\)=\varDelta_gS_g\(\xi\)+\frac{1}{3}S_g\(\xi\)^2
\end{equation}
and
\begin{equation}\label{AubinEq7}
A_g\(\xi\)=\frac{18\varDelta_gS_g\(\xi\)+8\left|\Ric_g\(\xi\)\right|_g^2-3\left|\Rm_g\(\xi\)\right|_g^2+5S_g\(\xi\)^2}{360N\(N+2\)}\,.
\end{equation}
The orthogonal decomposition of Riemann curvature is given by
\begin{equation}\label{AubinEq8}
\left|\Rm_g\(\xi\)\right|_g^2=\left|\Weyl_g\(\xi\)\right|_g^2+\frac{4}{N-2}\left|\E_g\(\xi\)\right|_g^2+\frac{2}{N\(N-1\)}S_g\(\xi\)^2,
\end{equation}
where $\Weyl_g$ is the Weyl curvature of $g$ and $\E_g=\Ric_g-\frac{S_g}{N} g$ is the traceless part of the Ricci curvature of $g$. Moreover, we get
\begin{equation}\label{AubinEq9}
\left|\Ric_g\(\xi\)\right|_g^2=\left|\E_g\(\xi\)\right|_g^2+\frac{1}{N}S_g\(\xi\)^2.
\end{equation}
By \eqref{llp} and \eqref{AubinEq5}, we compute
\begin{align} \label{AubinEq12}
&\int_M\left|\nabla U_{\mu,\xi}\right|_g^2dv_g=[N(N-2)]^{\frac{N-2}{2}} (N-2)^2\int_0^{\frac{r_0}{2}}\frac{\mu^{N-2}r^2}{\(\mu^2+r^2\)^N}\int_{\partial B_\xi\(r\)}d\sigma_gdr+\operatorname{O}\(\mu^{N-2}\)\\
&\quad=[N(N-2)]^{\frac{N-2}{2}} (N-2)^2 \omega_{N-1}\nonumber\\
&\qquad\times\int_0^{\frac{r_0}{2\mu}}\frac{r^{N+1}}{\(1+r^2\)^N}\(1-\frac{1}{6N}S_g\(\xi\) \mu^2 r^2+A_g\(\xi\) \mu^4 r^4+\operatorname{O}\(\mu^5 r^5\)\)dr+\operatorname{O}\(\mu^{N-2}\)\nonumber\\
&\quad= \frac{ [N(N-2)]^{\frac{N-2}{2}} (N-2)^2}{2}\omega_{N-1}\nonumber\\
&\qquad\times\(I_N^{\frac{N}{2}}-\frac{1}{6N}I_N^{\frac{N+2}{2}}S_g \(\xi\)\mu^2+
I_N^{\frac{N+4}{2}}A_g\(\xi\)\mu^4+\operatorname{O}\(I_N^{\frac{N+5}{2}} \mu^5+\mu^{N-2}\)\) \nonumber\\
&\quad=\left\{ \begin{aligned}
&K_N^{-N} \(1-\frac{N+2}{6N\(N-4\)}S_g \(\xi\)\mu^2\)
-  9216\ \omega_5 A_g\(\xi\)\mu^4 \ln \mu  +\operatorname{O}\(\mu^4\) &\hbox{if } N=6\\
&K_N^{-N} \(1-\frac{N+2}{6N\(N-4\)}S_g \(\xi\)\mu^2+\frac{\(N+2\)\(N+4\)}{\(N-4\)
\(N-6\)} A_g\(\xi\)\mu^4\)  +\operatorname{O}\(\mu^5\)& \hbox{if }N\geq 7
\end{aligned}\right. \nonumber
\end{align}
in view of \eqref{AubinEq16}. Since by \eqref{AubinEq11} there hold
$$I_{N-2}^{\frac{N-2}{2}}=\frac{4(N-1)(N-2)}{N(N-4)}I_{N}^{\frac{N}{2}}\quad\text{and}\quad
I_{N-2}^{\frac{N}{2}}=\left\{\begin{array}{ll}-2 \ln \mu+\operatorname{O}(1) &\hbox{if }N=6\\
\frac{4(N-1)(N-2)}{(N-4)(N-6)}I_N^{\frac{N}{2}}& \hbox{if }N\geq 7
\end{array}\right.$$
as $\mu \to 0$, by \eqref{AubinEq4} we compute
\begin{align}
&\int_M S_g U_{\mu,\xi}^2dv_g= [N(N-2)]^{\frac{N-2}{2}} \int_0^{\frac{r_0}{2}}\frac{\mu^{N-2}}{\(\mu^2+r^2\)^{N-2}}\int_{\partial B_\xi\(r\)}S_g d\sigma_gdr+\operatorname{O}\(\mu^{N-2}\)\nonumber\\
&\quad=[N(N-2)]^{\frac{N-2}{2}} \omega_{N-1}\mu^2\int_0^{\frac{r_0}{2\mu}}\frac{r^{N-1}}{\(1+r^2\)^{N-2}}\(S_g \(\xi\)-\frac{1}{2N}\varLambda_g\(\xi\)
\mu^2r^2+\operatorname{O}\(\mu^4r^4\)\)dr\nonumber\\
&\qquad+\operatorname{O}\(\mu^{N-2}\)\nonumber\\
&\quad=\frac{[N(N-2)]^{\frac{N-2}{2}}}{2}\omega_{N-1}\mu^2\(I_{N-2}^{\frac{N-2}{2}}S_g\(\xi\)
-\frac{1}{2N}I_{N-2}^{\frac{N}{2}}\varLambda_g\(\xi\)\mu^2+\operatorname{O}\(\mu^4 I_{N-2}^{\frac{N+2}{2}}+\mu^{N-2}\)\) \nonumber\\
&\quad=\left\{ \begin{aligned}
&\frac{ 5 K_6^{-6}}{12}
\mu^2 S_g\(\xi\)+48 \omega_5 \varLambda_g\(\xi\)\mu^4 \ln \mu+\operatorname{O}\(\mu^4\) &\hbox{if } N=6\\
&\frac{4 (N-1)K_N^{-N}}{N(N-2)(N-4)}
\mu^2\(S_g\(\xi\)-\frac{1}{2\(N-6\)}\varLambda_g\(\xi\)\mu^2\)+\operatorname{O}\(\mu^5\)& \hbox{if }N\geq 7
\end{aligned}\right.
\label{AubinEq13}
\end{align}
in view of \eqref{AubinEq16}. Similarly, by \eqref{AubinEq3}, we have that
\begin{equation}
\eps \int_MhU_{\mu,\xi}^2dv_g= 
\frac{4 (N-1)K_N^{-N}}{N(N-2)(N-4)}
h\(\xi\) \eps \mu^2+\operatorname{o}\(\eps \mu^2\) \label{AubinEq14}
\end{equation}
By \eqref{llp} and \eqref{AubinEq5}, we compute
\begin{align}
&\int_MU_{\mu,\xi}^{2^*}dv_g=[N(N-2)]^{\frac{N}{2}}\int_0^{\frac{r_0}{2}}\frac{\mu^N}{\(\mu^2+r^2\)^N}\int_{\partial B_\xi\(r\)}d\sigma_gdr+\operatorname{O}\(\mu^N\)\nonumber\\
&\quad=[N(N-2)]^{\frac{N}{2}}\omega_{N-1}\int_0^{\frac{r_0}{2\mu}}\frac{r^{N-1}}{\(1+r^2\)^N}\(1-\frac{1}{6N}S_g\(\xi\)\mu^2r^2+A_g\(\xi\)
\mu^4r^4\)dr+\operatorname{O}\(\mu^5\)\nonumber\\
&\quad=\frac{[N(N-2)]^{\frac{N}{2}}}{2}\omega_{N-1}\(I_N^{\frac{N-2}{2}}-\frac{1}{6N}I_N^{\frac{N}{2}}S_g\(\xi\)\mu^2+I_N^{\frac{N+2}{2}}
A_g\(\xi\)\mu^4\)+\operatorname{O}\(\mu^5\) \nonumber\\
&\quad=K_N^{-N} \(1-\frac{1}{6\(N-2\)}S_g\(\xi\)\mu^2+\frac{N\(N+2\)}{\(N-2\)\(N-4\)}A_g\(\xi\)\mu^4\)+\operatorname{O}\(\mu^5\)\label{AubinEq15}
\end{align}
in view of \eqref{AubinEq16}. Finally, the claimed expansions follow by \eqref{AubinEq12},  \eqref{AubinEq13},  \eqref{AubinEq14} and \eqref{AubinEq15} in view of \eqref{AubinEq6}--\eqref{AubinEq9}.
\end{proof}

\section{The Lyapunov-Schmidt reduction argument}\label{Sec4}

Since equation \eqref{Eq1} can be re-written as \eqref{Eq1b}, the function $u=\mathcal{W}_{\mu,\xi}+\phi$ does solve \eqref{Eq1} as soon as
\begin{equation}\label{eq1c}
\hat L_{\mu,\xi}(\phi)=-\mathcal{R}_{\mu,\xi}-N_{\mu,\xi}(\phi),
\end{equation}
where $\mathcal{R}_{\mu,\xi}$ is given in \eqref{rmx}, 
$$N_{\mu,\xi}(\phi)=-i^*\[\(\mathcal{W}_{\mu,\xi}+\phi\)_+^p-\(\mathcal{W}_{\mu,\xi}\)_+^p -p \(\mathcal{W}_{\mu,\xi}\)_+^{p-1} \phi \] $$
is the nonlinear term (quadratic in $\phi$) and
$$\begin{array}{lccl}\hat L_{\mu,\xi}: &H_g^1(M)& \to & H_g^1(M)\\
&\phi& \mapsto & \phi-i^*\[p \(\mathcal{W}_{\mu,\xi}\)_+^{p-1} \phi- \eps h \phi\]
\end{array}$$
is the linearized operator of \eqref{Eq1b} at $\mathcal{W}_{\mu,\xi}$.

\medskip
Since $\mathcal{W}_{\mu,\xi}$ is a small perturbation of $\mathcal{U}_{\mu,\xi}$, as $\eps, \mu \to 0$ the operator $\hat L_{\mu,\xi}$ in balls with radii of order $\mu$ looks pretty much as a scaling of the limiting operator $L_\infty:\ \Phi \to \Phi+(\Delta)^{-1} \[p U^{p-1} \Phi\]$, where $U$ is given in \eqref{defU}. It is well known (see Bianchi--Egnell~\cite{BiaEgn}) that 
$$\hbox{ker }L_\infty=\hbox{Span }\left\{\Phi^0,\Phi^1,\dots, \Phi^N \right\},$$
where
\begin{equation} \label{Eq1415}
\Phi^0(y)=\frac{1-|y|^2}{(1+|y|^2)^{\frac{N}{2}}}\, ,\qquad \Phi^i(y)=\frac{y^i}{(1+|y|^2)^{\frac{N}{2}}}\quad\forall \ i=1,\dots,N.
\end{equation}
Since there is no hope for the full invertibility of $\hat L_{\mu,\xi}$ in $H_g^1(M)$, let us introduce the ``asymptotic kernel" $K_{\mu,\xi}$ and its ``orthogonal space" $K^\perp_{\mu,\xi}$ as
$$K_{\mu,\xi}=\hbox{Span }\left\{ Z^0_{\mu,\xi},\dotsc, Z^N_{\mu,\xi} \right\}$$
and
$$K^\perp_{\mu,\xi}=\left\{\phi\in H^1_g\(M\):\ \int_M \mathcal{U}_{\mu,\xi}^{p-1} Z^i_{\mu,\xi} \phi d\mu_g =0\quad\forall\ i=0,\dotsc,N\right\},$$
where
$$Z^i_{\mu,\xi}(z)=\chi\(d_g(z,\xi)\) \mu^{\frac{2-N}{2}} \Phi^i \(\frac{\exp_\xi^{-1}(z)}{\mu}\)$$
for $i=0,\dotsc,N$, with $\Phi^i$ given by \eqref{Eq1415}. Letting $\Pi_{\mu,\xi}$ and $\Pi^\perp_{\mu,\xi}$ be the projectors of $H^1_g\(M\)$ onto the respective subspaces, equation \eqref{eq1c} is equivalent to solving 
\begin{align}\label{s1}
& L_{\mu,\xi}(\phi)= -\Pi^\perp_{\mu_\xi}\(\mathcal{R}_{\mu,\xi}+N_{\mu,\xi}(\phi)\),\\
\label{s2}
& \Pi _{\mu_\xi} \big(\hat L_{\mu,\xi}(\phi)\big)= -\Pi_{\mu_\xi}\(\mathcal{R}_{\mu,\xi}+N_{\mu,\xi}(\phi)\)
\end{align}
for some $\phi \in K^\perp_{\mu,\xi}$, where $L_{\mu,\xi}= \Pi^\perp_{\mu,\xi} \circ \hat L_{\mu,\xi}: K^\perp_{\mu,\xi} \to K^\perp_{\mu,\xi}$. First we can solve equation \eqref{s1}, a rather standard result in this context (see for example Musso--Pistoia~\cite{mupi4}):

\begin{lemma}\label{Lem1}
There exists a positive constant $C_0$ such that, for any $\eps,\mu$ small and any $\xi\in M$, there holds
$$\left\|L _{\mu,\xi}\(\phi\)\right\|  \ge C_0\left\|\phi\right\| $$
for all $\phi \in K^\perp_{\mu,\xi }$. As a consequence, \eqref{s1} admits a unique solution $\phi_{\mu,\xi} \in K^\perp_{\mu,\xi }$, which is continuously differentiable in $\mu$ and $\xi$, so that
\begin{equation}\label{Lem2Eq1tt}
\left\|\phi_{\mu,\xi} \right\|=
\left\{\begin{array}{ll} \operatorname{o}\big(\mu^2\sqrt{|\ln \mu|}+\sqrt \eps \mu \big) &\hbox{if } N=6\\
\operatorname{o}\(\mu^2+\sqrt \eps \mu\)&\hbox{if } N\geq 7. \end{array} \right.
\end{equation}
\end{lemma}

Let us just stress out that the estimate \eqref{Lem2Eq1tt} heavily depends on \eqref{Lem2Eq1}. The need of improving the ansatz in Section 2 comes out from getting the correct smallness rate of $\phi$ as expressed by \eqref{Lem2Eq1tt}. Finally, we have all the ingredients to prove our main result.

\begin{proof}[Proof of Theorem~\ref{Th}]
A first well known fact (see for example Musso--Pistoia~\cite{mupi4}) is the equivalence between equation \eqref{s2} and the search of critical points for
$$\mathcal{\widetilde J}_\eps(\mu,\xi)=J_\eps \(\mathcal{W}_{\mu,\xi}+\phi_{\mu,\xi}\),$$
where $\phi_{\mu,\xi}$ is given by Lemma~\ref{Lem1}.
We just need to prove that
\begin{equation}\label{energy1.1}
J_\eps  \(\mathcal{W}_{\mu,\xi}+\phi_{\mu,\xi}\)-J_\eps  \(\mathcal{W}_{\mu,\xi}\)=\left\{\begin{aligned}
&\operatorname{o}\(\mu^4 |\ln \mu|+\eps \mu^2\) &\hbox{if }N=6\\
&\operatorname{o}\(\mu^4+\eps \mu^2\)&\hbox{if }N\geq 7
\end{aligned} \right.
\end{equation}
as $\eps,\mu \to 0$. Indeed, we have that
\begin{align*}
&J_\eps  \(\mathcal{W}_{\mu,\xi}+\phi_{\mu,\xi}\)-J_\eps  \(\mathcal{W}_{\mu,\xi}\)=\int_M \(\<\nabla \mathcal{R}_{\mu,\xi},\nabla \phi_{\mu,\xi}\>_g+\alpha_N S_g \mathcal{R}_{\mu,\xi} \phi_{\mu,\xi}-\(W_{\mu,\xi}\)^p_+\phi_{\mu,\xi}\)  dv_g\nonumber\\ 
&\qquad+{1\over2}\int_M|\nabla \phi_{\mu,\xi}|_g^2dv_g+{1\over2}\int_M\(\alpha_N S_g+\eps   h\) \phi_{\mu,\xi} ^2 dv_g\nonumber\\ 
&\qquad-\frac{1}{p+1}
 \int_M   \[ \(\mathcal{W}_{\mu,\xi}+\phi_{\mu,\xi}\)_+^{p+1}-\(\mathcal{W}_{\mu,\xi}\)_+^{p+1}
 -(p+1)\(\mathcal{W}_{\mu,\xi}\)_+^p \phi_{\mu,\xi}\] dv_g\\
&\quad=\operatorname{O}\( \| \mathcal{R}_{\mu,\xi}\|\ \|\phi_{\mu,\xi}\|+  \|\phi_{\mu,\xi}\|^2+\int_M
\(\mathcal{W}_{\mu,\xi}\)_+^{p-1} \phi_{\mu,\xi}^2 dv_g+\int_M \phi_{\mu,\xi}^{p+1}dv_g \)\\
&\quad=\operatorname{O}\( \| \mathcal{R}_{\mu,\xi}\|\ \|\phi_{\mu,\xi}\|+  \|\phi_{\mu,\xi}\|^2\)
\end{align*}
by the Sobolev embedding $H_g^1(M)\hookrightarrow L^{p+1}(M)$ and the H\"older's inequality. By \eqref{Lem2Eq1} and \eqref{Lem2Eq1tt} we then deduce the validity of \eqref{energy1.1}. Setting
$$\mu(d)=d \left\{ \begin{array}{ll} l^{-1}(\eps)&\hbox{if }N=6\\
\sqrt \eps &\hbox{if }N\geq 7, \end{array}\right. $$
where $l:(0 , e^{-\frac{1}{2}})\to (0,\frac{e^{-1}}{2})$ is defined as $l(\mu)=-\mu^2 \ln \mu$, by Proposition~\ref{Pr2} 
 and \eqref{energy1.1} we deduce the following asymptotic estimates:
 $$\mathcal{J}(d,\xi):=\frac{\mathcal{\widetilde J}_\eps(\mu(d),\xi)-K_N^{-N}}{\eps^2}\(\ln l^{-1}(\eps)\)^\gamma=c_2 d^2 h(\xi)-c_3 d^4 \left|\Weyl_g(\xi)\right|^2_g+\operatorname{o}(1)$$
as $\eps \to 0$, uniformly with respect to $\xi \in M$ and to $d$ in compact subsets of $(0,\infty)$, where $c_2,c_3>0$ are suitable constants, $\gamma=1$ when $N=6$ and $\gamma=0$ when $N\geq 7$. Letting $\mathcal{D} \subset (0,\infty) \times M$ be a $C^0$--{\it stable critical set} of $\widetilde E$ and $U$ be a compact neighborhood of $\mathcal{D}$ in $(0,\infty) \times M$, by the definition of stability it follows that $\mathcal{J}$ has a critical point $\(d_\varepsilon,\xi_\varepsilon\)\in  U \subset (0,\infty) \times M$, for $\varepsilon$ small. Up to a subsequence and taking $U$ smaller and smaller, we can assume that $\(d_\varepsilon,\xi_\varepsilon\) \to \(t_0,\xi_0\)$ as $\varepsilon \to 0$ with $\xi_0 \in \pi(\mathcal{D})$. By elliptic regularity theory $u_\varepsilon=\mathcal{W}_{\mu(d_\eps),\xi_\eps}+\phi_{\mu(d_\eps),\xi_\eps}$ is a solution of \eqref{Eq1}. Since $\xi_\varepsilon\to\xi_0$ and $\left\|\phi_{\mu(d_\varepsilon),\xi_\varepsilon} \right\| \to0$ as $\varepsilon \to 0$, it is easily seen that $u_\varepsilon>0$ and $u_\varepsilon^{2^*} \rightharpoonup K_N^{-N}\delta_{\xi_0}$ in the measures sense as $\varepsilon \to 0$ (see for example Rey~\cite{Rey}), where $\delta_\xi$ denotes the Dirac mass measure at $\xi$. From very basic facts concerning the asymptotic analysis of solutions of Yamabe-type equations (see for example Druet--Hebey~\cite{DruHeb2} and Druet--Hebey--Robert\cite{DruHebRob}), we get that the family $(u_\varepsilon)_\varepsilon$ blows up at the point $\xi_0$ as $\varepsilon \to 0$.
\end{proof}

\end{document}